\documentclass[12pt]{amsart}
\usepackage[left=2.5cm,top=2.5cm,right=2.5cm]{geometry}
\usepackage{verbatim}
\usepackage{amsmath}
\usepackage{amssymb}
\usepackage{amscd}
\usepackage{color}
\usepackage{url}
\usepackage{graphicx}
\usepackage[
backref,
pdfauthor={Farbod Shokrieh}, 
]{hyperref}
\usepackage[
alphabetic,
backrefs,
msc-links,
nobysame,
lite,
]{amsrefs} 

\numberwithin{equation}{section}

\newtheorem{theorem}[equation]{Theorem}
\newtheorem{lemma}[equation]{Lemma}

\newtheorem{proposition}[equation]{Proposition}
\newtheorem{corollary}[equation]{Corollary}
\newtheorem{definition}[equation]{Definition}
\newtheorem{example}[equation]{Example}

\theoremstyle{remark}
\newtheorem{remark}[equation]{Remark}

\renewcommand{\O}{{\mathcal O}}

\newcommand{\C}{{\mathcal C}}
\newcommand{\D}{{\mathcal D}}

\newcommand{\Cf}{{\mathfrak C}}
\newcommand{\Bf}{{\mathfrak B}}
\newcommand{\Df}{{\mathfrak D}}
\newcommand{\Lf}{{\mathfrak L}}

\newcommand{\cb}{{\mathbf c}}

\newcommand{\Cb}{{\mathbf C}}

\def\RR{{\mathbb R}}
\def\ZZ{{\mathbb Z}}

\newcommand{\Div}{\operatorname{Div}}

\newcommand{\Jac}{\operatorname{Jac}}

\newcommand{\Prin}{\operatorname{Prin}}

\newcommand{\Pic}{\operatorname{Pic}}

\newcommand{\outdeg}{{\operatorname{outdeg}}}

\newcommand{\Vol}{{\rm Vol}}
\newcommand{\indeg}{{\rm indeg}}
\newcommand{\supp}{{\rm Supp}}

\begin{document}

\title[Canonical representatives for divisor classes on tropical curves]
    {Canonical representatives for divisor classes on tropical curves
    and the Matrix-Tree Theorem}


\thanks{Matthew Baker and Farbod Shokrieh were partially supported by NSF grant DMS-1201473. Greg Kuperberg was partially supported by NSF grant CCF-1013079.  We enthusiastically thank Spencer Backman and the anonymous referee for their helpful comments.}

\author{Yang An}
\email{yangan@math.columbia.edu}
\address{Columbia University}
\author{Matthew Baker}
\email{mbaker@math.gatech.edu}
\address{Georgia Institute of Technology}
\author{Greg Kuperberg}
\email{greg@math.ucdavis.edu}
\address{University of California, Davis}
\author{Farbod Shokrieh}
\email{farbod@math.cornell.edu}
\address{Cornell University}

\begin{abstract}
Let $\Gamma$ be a compact tropical curve (or \emph{metric graph}) of genus
$g$.  Using the theory of tropical theta functions, Mikhalkin and Zharkov
proved that there is a canonical effective representative (called a \emph{break divisor}) for each linear
equivalence class of divisors of degree $g$ on $\Gamma$.
We present a new combinatorial proof of the fact that there is a unique
break divisor in each equivalence class, establishing in the process an ``integral'' version of this result which is of
independent interest.
As an application,
we provide a ``geometric proof'' of (a dual version of) Kirchhoff's celebrated Matrix-Tree Theorem.
Indeed, we show that each weighted graph model $G$ for $\Gamma$ gives rise
to a canonical polyhedral decomposition of the $g$-dimensional real torus
$\Pic^g(\Gamma)$ into parallelotopes $C_T$, one for each spanning tree $T$ of
$G$, and the dual Kirchhoff theorem becomes the statement that the volume of
$\Pic^g(\Gamma)$ is the sum of the volumes of the cells in the decomposition.
\end{abstract}

\maketitle


\section{Introduction}
\label{sec:Introduction}

Let $\Gamma$ be a compact tropical curve (or \emph{metric graph}) of
genus $g$.  There is a canonical continuous map $\pi : \Div^g_+(\Gamma)
\to \Pic^g(\Gamma)$ taking an effective divisor of degree $g$ on $\Gamma$
to its linear equivalence class.  Using tropical theta functions, Mikhalkin
and Zharkov \cite{MK08} showed that there is a canonical continuous section
$\sigma$ to the map $\pi$.  In particular, every divisor class of degree
$g$ has a canonical effective representative.  (This is in sharp contrast
to the situation for compact Riemann surfaces, where the analogous map
$\pi$ does not admit such a section.)  Although not stated explicitly in that paper, one can deduce easily from the results in
\cite{MK08} that the image of $\sigma$ is the set of \emph{break divisors}
in $\Div^g_+(\Gamma)$ (combine Theorem 6.5, Corollary 6.6, and Lemma 8.3 from \cite{MK08}).

\medskip

In this paper, we study break divisors in detail
and give some applications.  One application is a new combinatorial
proof of the existence of the section $\sigma$ which does not make use of
tropical theta functions; this proof has the advantage that it yields an integral version of the Mikhalkin--Zharkov theorem.
Another application is a ``geometrization'' of Kirchhoff's
celebrated Matrix-Tree Theorem: we show that for each weighted graph model
$G$ for $\Gamma$, there is a canonical polyhedral decomposition of the
$g$-dimensional real torus $\Pic^g(\Gamma)$ into parallelotopes $C_T$,
one for each spanning tree $T$ of $G$; from this point of view Kirchhoff's
theorem (or rather its \emph{matroid dual}) becomes the statement that
the volume of $\Pic^g(\Gamma)$ is the sum of the volumes of the cells in
the decomposition.

\medskip

In order to define break divisors, it is convenient to fix a model $G$ for $\Gamma$.
For each spanning tree $T$ of $G$, let $\Sigma_T$ (resp. $\Sigma_T^\circ$) be the image of the canonical map
$\prod_{e \not\in T} {\bar e} \to \Div^g_+(\Gamma)$ (resp. $\prod_{e \not\in T} e^\circ \to \Div^g_+(\Gamma)$) sending
$(p_1,\ldots,p_g)$ to $(p_1)+ \cdots + (p_g)$.  (Here $\bar{e}$ denotes a closed edge and $e^\circ$ denotes the corresponding open edge with the
endpoints removed.)
We call $\Sigma := \bigcup_T \Sigma_T$ the set of \emph{break divisors} on $\Gamma$.
It does not depend on the choice of the model $G$.
Our first main result is a new combinatorial proof of the following theorem of Mikhalkin and Zharkov:

\begin{theorem}
\label{thm:mainthm1}
Every degree $g$ divisor on $\Gamma$ is linearly equivalent to a unique break divisor.
\end{theorem}

Since $\Sigma$ is compact, and a continuous bijection between compact Hausdorff spaces is a homeomorphism, the theorem implies that
$\pi$ induces a homeomorphism from $\Sigma$ onto its image.  The canonical section $\sigma$ is the inverse of this homeomorphism.

\medskip

Our combinatorial proof of Theorem~\ref{thm:mainthm1}
is based on a study of \emph{orientable divisors}.
If $\O$ is a (not necessarily acyclic) orientation of $\Gamma$, we define a corresponding divisor $D_\O$ of degree $g-1$ by the formula
\[
D_\O = \sum_{p \in \Gamma} \left( {\rm indeg}_\O(p) - 1 \right)(p).
\]
We call divisors of this form \emph{orientable}.  We will show by a constructive procedure that every divisor of degree $g-1$ is linearly equivalent to an orientable divisor.
More precisely, fix $q \in \Gamma$.  We say that an orientation is \emph{$q$-connected} if there is a directed path from $q$ to $p$ for every $p \in \Gamma$.
A divisor of the form $D_\O$, with $\O$ a $q$-connected orientation, is called \emph{$q$-orientable}.  We will prove:

\begin{theorem} \label{thm:mainthm2}
Fix $q \in \Gamma$.  Every divisor of degree $g-1$ on $\Gamma$ is linearly equivalent to a unique $q$-orientable divisor.
\end{theorem}

There is a close connection between break divisors and $q$-orientable divisors.  Indeed, we will see that the map sending a degree $g$ divisor $D$ to the
degree $g-1$ divisor $D - (q)$ induces a bijection between break divisors and $q$-orientable divisors.
Using this observation, one deduces in a completely formal way that Theorem~\ref{thm:mainthm1} and Theorem~\ref{thm:mainthm2} are in fact equivalent.

\medskip

One advantage of our approach to Theorems \ref{thm:mainthm1} and \ref{thm:mainthm2} is that it enables us to prove an \emph{integral version} of
Theorem~\ref{thm:mainthm1} for finite unweighted graphs $G$ (or equivalently, finite weighted graphs in which all edges have length $1$).
Indeed, if we define an \emph{integral break divisor} to be a break divisor supported on the vertices of $G$, then
the constructive procedure furnished by our proof of Theorem~\ref{thm:mainthm2} shows:

\begin{theorem} \label{thm:mainthm1bis}
Every degree $g$ divisor supported on the vertices of $G$ is linearly equivalent to a unique integral break divisor.
\end{theorem}

As a consequence of Theorem~\ref{thm:mainthm1bis}, the set ${\rm Pic}^g(G)$ of divisors of degree $g$ on $G$ modulo linear equivalence is canonically in bijection with the set of integral break divisors.  In particular, the number of integral break divisors is equal to the number of spanning trees of $G$.

\medskip

We now turn to a more detailed discussion of the connection between break divisors, polyhedral decompositions of $\Pic^g(\Gamma)$, and Kirchhoff's Matrix-Tree Theorem.
Let $G$ be a weighted graph with associated metric graph $\Gamma$.
Recall that $\Sigma := \bigcup_T \Sigma_T$, where the union is over all spanning trees $T$ of $G$.
If we define $\C_T = \pi(\Sigma_T)$, then clearly
$\Pic^g(\Gamma) = \bigcup_T \C_T$ by Theorem~\ref{thm:mainthm1}.
It turns out that each cell $\C_T$ is a parallelotope and that if $T \neq T'$ then the relative interiors of $\C_T$ and $\C_{T'}$ are disjoint.
Thus $\Pic^g(\Gamma)$ has a polyhedral decomposition depending only on the choice of the model $G$.
It is not hard to check that if $G'$ is a refinement of $G$, then the cell decomposition coming from $G'$ is a refinement of the cell decomposition coming from $G$.

\medskip

\begin{example}
Let $\Gamma$ be the metric graph consisting of two vertices joined by three edges of lengths $2$, $1$, and $2$, respectively. Fix a model $G$ for $\Gamma$ in which all edge lengths have length $1$.  In Figure~\ref{ex1} we have listed all spanning trees of $G$ and the corresponding cell decomposition of $\Pic^2(\Gamma)$. We have labeled the center of each cell with its corresponding break divisor.

\begin{figure}[h!]
\begin{center}
    \includegraphics[width=0.5\textwidth]{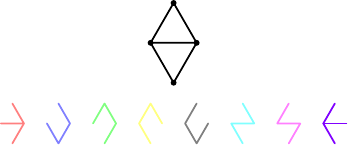}
\end{center}
  \vspace{+20pt}
\begin{center}
    \includegraphics[width=0.8\textwidth]{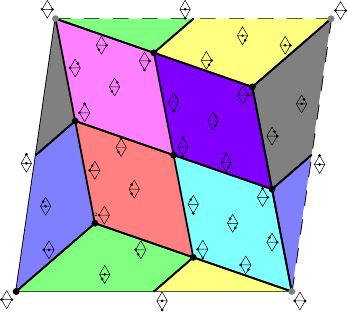}
\end{center}
  \caption{A fixed model for the metric graph $\Gamma$ and the corresponding decomposition of $\Pic^2(\Gamma)$.}
  \label{ex1}
\end{figure}

\end{example}

Since the canonical map from $\Sigma = \bigcup \Sigma_T$ to $\Pic^g(\Gamma)$ is a homeomorphism
and $\C_T$ is the projection to $\Pic^g(\Gamma)$ of the cube $\Sigma_T$,
it follows easily that if $T,T'$ are spanning trees of $G$ then the dimension of
$\C_T \cap \C_{T'} \subset \Pic^g(\Gamma)$ is equal to the dimension of $\Sigma_T \cap \Sigma_{T'}$ inside
$\prod_{e \in E} {\bar e}$.  In particular,
$\C_T \cap \C_{T'}$ is non-empty if and only if
$\prod_{e \not\in T} {\bar e} \cap \prod_{e \not\in T'} {\bar e}$ is non-empty.

\medskip

Moreover, recall from \cite{MK08} that $\Pic^g(\Gamma)$ is canonically a principal homogeneous space for the Picard group $\Pic^0(\Gamma)$
and that there is a canonical isomorphism (the ``tropical Abel-Jacobi map'') between $\Pic^0(\Gamma)$ and $\Jac(\Gamma) = H_1(\Gamma,\RR) / H_1(\Gamma,\ZZ)$,
which is a real torus of dimension $g$.  The intersection form on $H_1(\Gamma,\ZZ)$ gives rise to a
canonical translation-invariant Riemannian metric on $\Jac(\Gamma)$, and via translation-invariance $\Pic^g(\Gamma)$ is equipped with a canonical metric as well.
In particular, the volume of $\Pic^g(\Gamma)$ and of each of the cells $C_T$ is well-defined.

\begin{theorem} \label{thm:mainthm3}
For each spanning tree $T$ of $G$, the volume of the parallelotope $C_T$ is $\frac{w(T)}{{\rm Vol}(\Jac(\Gamma))}$,
where $w(T) := \prod_{e \not \in E(T)}{\ell(e)}$ is the product of the lengths all edges of $G$ not in $T$.
Moreover, the volume of $\Jac(\Gamma)$ is $\sqrt{\det(\Cf \Cf^t)}$, where $\Cf$ is any matrix whose rows form a $\ZZ$-basis for $H_1(G,\ZZ)$.
\end{theorem}

Since distinct cells $C_T$ intersect in positive codimension, we have $\Vol(\Jac(\Gamma)) = \sum_T \Vol(C_T)$ and thus
Theorem~\ref{thm:mainthm3} implies the following \emph{dual version of Kirchhoff's Matrix-Tree Theorem}:

\begin{corollary} \label{cor:Kirchhoff}
For any weighted graph $G$,
\[
\det(\Cf \Cf^t) = \sum_T w(T).
\]
\end{corollary}

The usual version of Kirchhoff's Matrix-Tree Theorem is (a special case of) the dual statement that for any weighted graph $G$, we have
\[
\det(\Bf \Bf^t) = \sum_T w'(T),
\]
where $w'(T) := \prod_{e \in E(T)}{\ell(e)}$ is the product of the lengths all edges of $G$ in $T$ and $\Bf$ is any matrix whose rows form a $\ZZ$-basis for
the \emph{cocycle lattice} of $G$ (which is the intersection of $C_1(G,\ZZ)$ with the orthogonal complement of $H_1(G,\RR)$ in $C_1(G,\RR)$).

\medskip

Note that the dual version of Kirchhoff's theorem, like the cycle lattice $H_1(G,\ZZ)$, is unchanged if we replace $G$ by a different model $G'$ for the same
metric graph $\Gamma$.  This is not true of the usual version of Kirchhoff's theorem, or of the cocycle lattice.

\section{Definitions and background}

\subsection{Metric graphs}

\begin{definition}
A metric graph (or abstract tropical curve) $\Gamma$ is a compact connected metric space such that every point $p \in \Gamma$ has a neighborhood isometric to a star-shaped set.
\end{definition}

For the purposes of this paper, we define a \textit{weighted graph} to be a finite, connected graph $G$ with no loop edges,
together with a collection of positive weights $L_e$ (which we also call \emph{lengths}), one for each edge.

A weighted graph $G$ gives rise to a metric space $\Gamma$ in the following way. To each
edge $e$, associate a line segment of length $L_e$, and identify the ends of distinct
line segments if they correspond to the same vertex of $G$. The points of these line
segments are the points of $\Gamma$. We call $G$ a $\textit{model}$ for $\Gamma$. The
distance between two points $x$ and $y$ in $\Gamma$ is defined to be the length of the
shortest path between them.

It is easy to check that every metric graph arises from this construction, though the weighted graph model is not unique.
Write $G \sim G'$ if the two weighted graphs $G, G'$ admit a \textit{common
refinement}, where we \textit{refine} a weighted graph by subdividing its edges in a
manner that preserves the total length. This yields an equivalence relation on the collection of weighted graphs, and two
weighted graphs are equivalent if and only if their associated metric graphs are isometric.
There is thus a bijective correspondence between metric graphs and equivalence classes of
weighted graphs (see \cite{BakerFaber11} for details).

\subsection{Reduced divisors}
Let  $\Div(\Gamma)$ (the group of \emph{divisors} on $\Gamma$) be the free abelian group on $\Gamma$. An element of $\Div(\Gamma)$ is of the form $D= \sum_{p\in \Gamma} a_p(p)$, where $a_p \in \ZZ$ and all but finitely many of $a_p$'s are zero. The degree of $D$ is by definition $\deg(D)=\sum_{p \in \Gamma} {a_p}$. Let $\Div^0 (\Gamma)$  be  the subgroup of  divisors  of  degree zero  on  $\Gamma$. A  divisor  $D  =  \sum a_p(p)$  is  called \emph{effective} if $a_p \geq 0$  for  all $p$, and is called \emph{effective outside $q$}  if $a_p \geq 0$ for all $p  \ne  q$. The support of $D$ is by definition $\supp(D)=\{p \in \Gamma \, | \, D(p) \ne 0 \}$.

\medskip

Let $R(\Gamma)$ be the group consisting of continuous piecewise affine functions with  integer slopes. This can be viewed as the space of \emph{tropical rational functions} on $\Gamma$, cf. \cite{GathmannKerber, MK08}.  Let $\Delta$ be the Laplacian operator on $\Gamma$; for $f \in R(\Gamma)$ we have
\[
\Delta (f) =  \sum_{p \in \Gamma}{\sigma_{p} (f) (p),}
\]	
where $-\sigma_p(f)$ is the sum of the slopes of $f$ in all tangent directions emanating from $p$.
(A \emph{tangent direction} at a point $p \in \Gamma$ is an equivalence class of paths emanating from $p$, where two paths are equivalent if they share a common initial segment.)
Let  $\Prin(\Gamma)$ be the subgroup $\{ \Delta (f) \, | \,  f \in  R(\Gamma)\}$  of  $\Div^0 (\Gamma)$ consisting of \emph{principal divisors}. We write $D \sim D'$ if $D - D'$ belongs to $\Prin(\Gamma)$ and say that $D$ and $D'$ are \emph{linearly equivalent}. For $D \in \Div(\Gamma)$, we define the complete linear system $|D|$ to be the set of all  effective divisors $E$ equivalent to $D$,  i.e., $|D|=  \{E  \in \Div(\Gamma)  \, | \, E  \ge 0,
E  \sim D\}$. Similarly, we define $|D|_q$ to be the set of divisors equivalent to $D$ which are effective outside $q$:
\[
|D|_q = \{E \in \Div(\Gamma) \, | \, E(p) \geq 0 , \forall p \ne q,  E \sim D\} \ .
\]

\begin{remark}
Given an effective divisor $D$, it is customary to think of $D(v)$ as the number of ``chips'' placed at the point $v \in \Gamma$. For a subset $X$ of $\Gamma$ consisting of a finite union of closed intervals, one can construct a rational function $f \in R(\Gamma)$ which is $0$ on $X$ and $\epsilon$ outside an $\epsilon$-neighborhood of $X$ (for some sufficienly small $\epsilon$), with slope $1$ in each outgoing direction from $X$. Replacing $D$ with $D+\Delta(f)$ has the effect of moving a chip a distance of $\epsilon$ along each outgoing direction from $X$, and is often called ``firing'' the subset $X$.  Every element of $R(\Gamma)$ can be written as a finite integer-affine combination of functions of this
form, and therefore one can describe linear equivalence of divisors on $\Gamma$ in terms of chip firing.
\end{remark}

\begin{definition} \label{def:reduced}
Fix $q  \in \Gamma$. A divisor $D$ on $\Gamma$ is called \emph{$q$-reduced} if it satisfies the following two conditions:
\begin{itemize}
 \item[(R1)]  $D$ is effective outside $q$.
 \item[(R2)]  For every closed connected set $X  \subseteq \Gamma$ not containing $q$, there exists a point $p \in \partial X$  such that $D(p) < \outdeg_X (p)$.
\end{itemize}
\end{definition}

The following is the metric graph analogue of \cite[Lemma 4.11]{FarbodMatt12}.
\begin{lemma} [Principle of least action] \label{LeastActLemma}
 If  $D$  is  $q$-reduced  and  $f  \in R(\Gamma)$  is  a  tropical rational  function with $D + \Delta(f) \in |D|_q$, then $f$  has a global minimum at $q$.
\end{lemma}

\begin{proof}	
Consider the set of points where  $f$ achieves its (global) minimum value. If this set is not $\{q\}$ then we may find a closed connected component $X$ not containing $q$.  By Definition~\ref{def:reduced} there  exists  $p \in \partial X$ such that $D(p)  <  \outdeg_X(p)$. On the  other hand, we have $\Delta(f)(p) < -\outdeg_X (p)$, and thus $(D +\Delta(f)) (p) < 0$, contradicting the assumption that $D + \Delta(f) \in |D|_q $.
\end{proof}

The importance of reduced divisors is given by the following theorem (see, e.g., \cite{Hladkyetal, Luoye}).
	
\begin{theorem} \label{UniquenessThm}
 Fix $q \in \Gamma$. There is a unique $q$-reduced divisor in each linear equivalence class of divisors on $\Gamma$.
\end{theorem}

\subsection{Tropical Picard group, Jacobian, and Abel-Jacobi map}

The (degree $0$) {\bf Picard group} of $\Gamma$ is by definition the quotient
\[
\Pic^0(\Gamma):= \Div^0(\Gamma) / \Prin(\Gamma) \ .
\]
More generally, for $d \geq 0$, let $\Div^d (\Gamma)$  be  the subset of  divisors of degree $d$ on $\Gamma$ and define
\[
\Pic^d(\Gamma):= \Div^d(\Gamma) / \Prin(\Gamma) \ .
\]
The set $\Pic^d(\Gamma)$ is canonically a $\Pic^0(\Gamma)$-torsor.

\medskip

The tropical Abel-Jacobi theory identifies (as topological groups) $\Pic^0(\Gamma)$ with the {\bf Jacobian torus}
\[
\Jac(\Gamma)=\Omega^{\ast}(\Gamma) / H_1(\Gamma, \ZZ) \ .
\]
 Here $\Omega(\Gamma)$ denotes the space of \emph{harmonic $1$-forms} on $\Gamma$. A harmonic $1$-form on $\Gamma$ is obtained by assigning a real-valued slope to each edge in $\Gamma$ in such a way that the sum of the incoming slopes is zero at every vertex. The homology group $H_1(\Gamma, \ZZ)$ embeds as a lattice in $\Omega^{\ast}(\Gamma)$ (the dual vector space of $\Omega(\Gamma)$) by integration of $1$-forms along $1$-cycles.  There is a canonical identification of
$\Omega(\Gamma)^{\ast}$ with the singular cohomology space $H_1(\Gamma,\RR)$, so we will often write
\[
\Jac(\Gamma)= H_1(\Gamma, \RR) / H_1(\Gamma, \ZZ) \ .
\]

The vector space $H_1(\Gamma, \RR)$ is equipped with a natural translation-invariant Riemannian metric which induces a canonical metric, and in particular
a canonical volume form, on the quotient torus $\Jac(\Gamma)$.

\medskip

For each positive integer $d$, the corresponding {\bf Abel-Jacobi map}
\[
S^{(d)}: \Div_{+}^d(\Gamma) \rightarrow \Pic^d(\Gamma)
\]
sends a divisor $D \in \Div_{+}^d(\Gamma)$ to the divisor class $[D] \in \Pic^d(\Gamma)$.

We will denote $S^{(g)}$ by $\pi$.

Fixing a basepoint $q \in \Gamma$, the torus $\Pic^d(\Gamma)$ is identified with the group $\Pic^0(\Gamma)$ and $S^{(d)}$ is identified with the map
\[
S^{(d)}_q: \Div_{+}^d(\Gamma) \rightarrow \Pic^0(\Gamma)
\]
taking a divisor $D \in \Div_{+}^d(\Gamma)$ to the divisor class $[D-d(q)] \in \Pic^0(\Gamma)$.

All maps $S^{(d)}$ are \emph{piecewise linear} in the appropriate sense; fixing the basepoint $q$, the real vector space $\Omega^{\ast}(\Gamma)$ is identified with the universal cover of $\Pic^d(\Gamma)$, and the restriction of the Abel-Jacobi map to any contractible subset factors through a piecewise-linear map to $\Omega^{\ast}(\Gamma)$. Note that $\Div_{+}^d(\Gamma)$ is endowed with a natural integral affine structure induced from $\Gamma$ (see, e.g., \cite[Section 2.1]{Amini}).

\medskip

\section{Spanning trees and divisors}

\subsection{Break divisors, break pairs, fundamental domains, and orientations}

Recall from Section~\ref{sec:Introduction} that a \emph{break divisor} on $\Gamma$ is a divisor of the form $(p_1)+ \cdots + (p_g)$, where $p_i \in {\bar e_i}$ and
$\Gamma \backslash T = \bigcup_{i=1}^g e_i^\circ$ for some spanning tree $T$ of $G$.  (Here $G$ is any fixed model of $\Gamma$).
See Figure~\ref{ex01}.

Here is another way to characterize break divisors, following \cite{MK08}.
If $p$ is a vertex of some model $G$ for $\Gamma$, there is a natural bijection between tangent directions at $p$ and edges of $G$ incident to $p$.
For $\varepsilon > 0$ sufficiently small, there is a well-defined point $p+\varepsilon \eta$ at distance $\varepsilon$ from $p$ in the direction $\eta$.
Let $p_1,\ldots,p_g$ be (not necessarily distinct) points on $\Gamma$ and for each $i$ let $\eta_i$ be a tangent direction at $p_i$.
If $\Gamma \backslash \{ q_1,\ldots,q_g \}$ is a tree (i.e., is connected and simply connected) for $\varepsilon > 0$ sufficiently small, where $q_i = p_i+\varepsilon \eta_i$, we call the collection $\{ (p_1, \eta_1),\ldots,(p_g, \eta_g) \}$ a \emph{break set} for $\Gamma$.
See Figure~\ref{ex02}.

\begin{lemma} \label{lem:break_div_vs_pair}
If $\{ (p_1, \eta_1),\ldots,(p_g, \eta_g) \}$ is a break set then $(p_1)+\cdots+(p_g)$ is a break divisor.
Conversely, if $(p_1)+\cdots+(p_g)$ is a break divisor then there exist (not necessarily unique) tangent directions $\eta_i$ at $p_i$
such that $\{ (p_1, \eta_1),\ldots,(p_g, \eta_g) \}$ is a break set.
\end{lemma}

\begin{proof}
Let $\{ (p_1, \eta_1),\ldots,(p_g, \eta_g) \}$ be a break set and let $\varepsilon > 0$ be small enough so that $\Gamma \backslash \{ q_1,\ldots,q_g \}$, where $q_i = p_i+\varepsilon \eta_i$, is a tree and all $q_i$'s have valence $2$. Fix a model $G$ for $\Gamma$ such that $ p_i  \in V(G)$ and $q_i \not \in V(G)$ for all $i$. It follows that $\Gamma \backslash \{ q_1,\ldots,q_g \}$ contains a spanning tree $T$ of $G$, and that $p_i \in {\bar e_i}$, where
$\Gamma \backslash T = \bigcup_{i=1}^g e_i^\circ$.

\medskip

Conversely, assume $G$ is a model of $\Gamma$ and, for some spanning tree $T$ of $G$, we have $\Gamma \backslash T = \bigcup_{i=1}^g e_i^\circ$ and we are given a break divisor $(p_1)+ \cdots + (p_g)$, where $p_i \in {\bar e_i}$. If $p_i \in  e_i^\circ$ we may choose $\eta_i$ to be either of the two tangent directions at $p_i$. If $p_i \in \partial e_i= {\bar e_i} \backslash e_i^\circ $ then we choose the tangent direction directed towards $e_i^\circ$.
\end{proof}

Because of the lemma, we will sometimes denote a break set by $(D,\eta)$, where $D = (p_1)+\cdots+(p_g)$ is a break divisor and $\eta = \{ \eta_1 , \ldots, \eta_g \}$ is the
corresponding set of \emph{break directions}.

\begin{remark}
\label{FundDomainRemark}
If we think of a metric graph as made of wires, one can think of a break set as a rule for snipping the wires so that the resulting network is connected and simply
connected.  More formally, define a \emph{fundamental domain} for $\Gamma$ to be a pair $(F,\psi)$ consisting of a (not necessarily compact) connected and simply connected topological space $F$ and a continuous bijection $\psi : F \to \Gamma$.  We identify two fundamental domains $(F,\psi)$ and $(F',\psi')$ if there is a homeomorphism $\phi : F \to F'$
such that $\psi = \psi' \circ \phi$.  Then one can show that there is a natural bijection between fundamental domains for $\Gamma$ and break sets.
(This is the point of view taken in \cite[\S{4.5}]{MK08}.) See Figure~\ref{ex03}.
\end{remark}

\begin{figure}[h!]
\begin{center}
\includegraphics[totalheight=0.18\textheight]{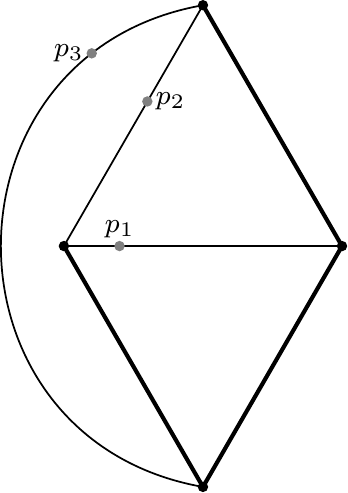}
\includegraphics[totalheight=0.18\textheight]{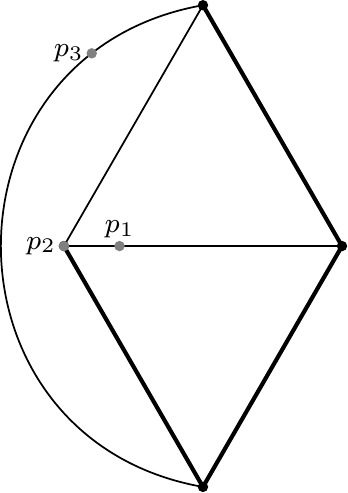}
\includegraphics[totalheight=0.18\textheight]{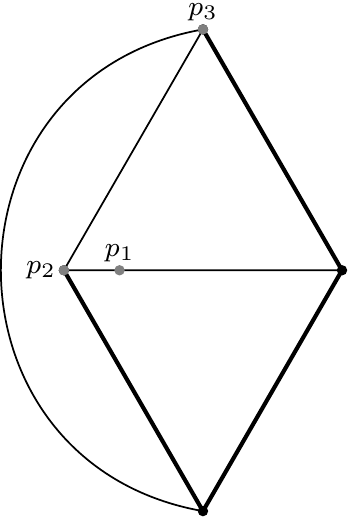}
\includegraphics[totalheight=0.18\textheight]{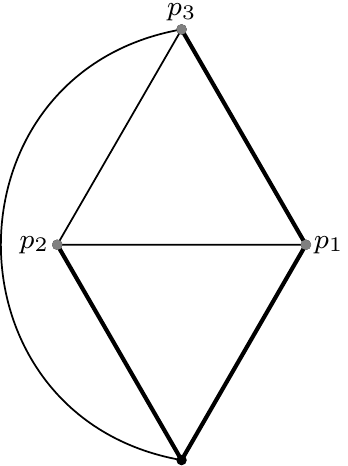}
\includegraphics[totalheight=0.18\textheight]{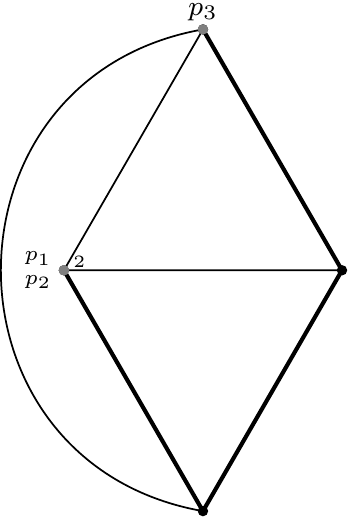}
\end{center}
  \caption{Examples of break divisors.}
  \label{ex01}
\vspace{+20pt}
\begin{center}
\includegraphics[totalheight=0.18\textheight]{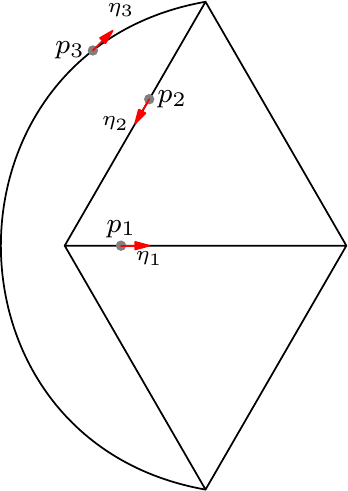}
\includegraphics[totalheight=0.18\textheight]{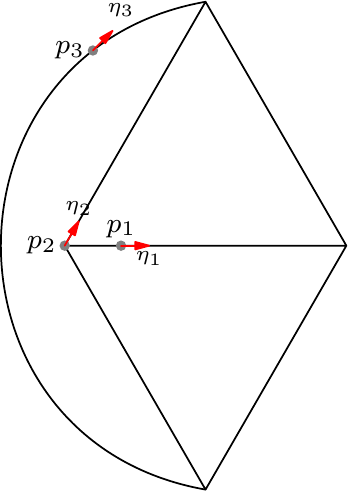}
\includegraphics[totalheight=0.18\textheight]{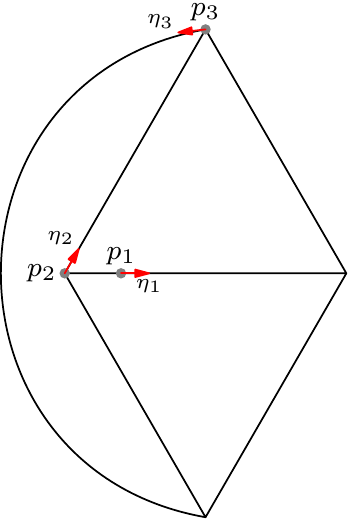}
\includegraphics[totalheight=0.18\textheight]{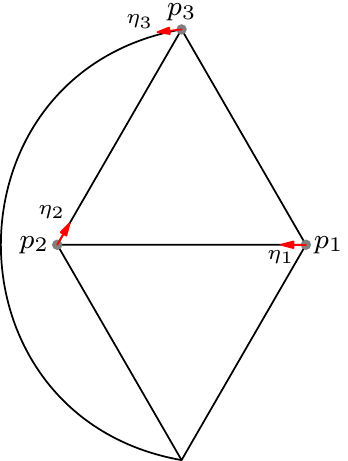}
\includegraphics[totalheight=0.18\textheight]{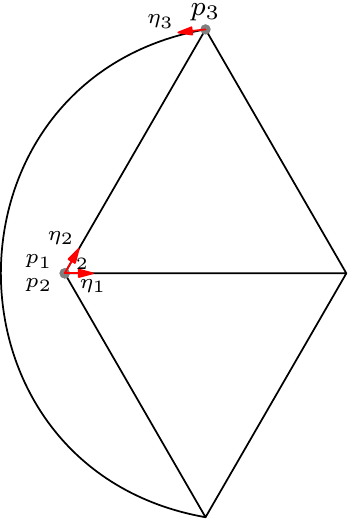}
\end{center}
  \caption{Examples of break sets.}
  \label{ex02}
\vspace{+20pt}
\begin{center}
\includegraphics[totalheight=0.18\textheight]{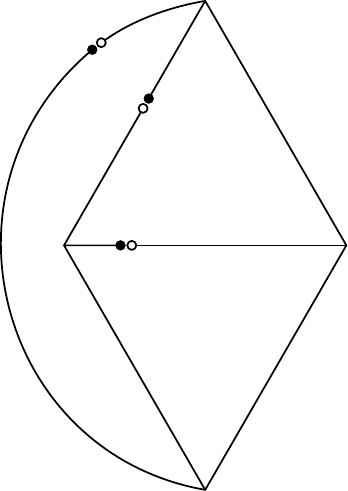}
\includegraphics[totalheight=0.18\textheight]{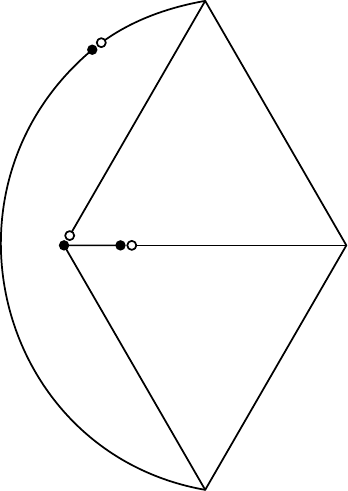}
\includegraphics[totalheight=0.18\textheight]{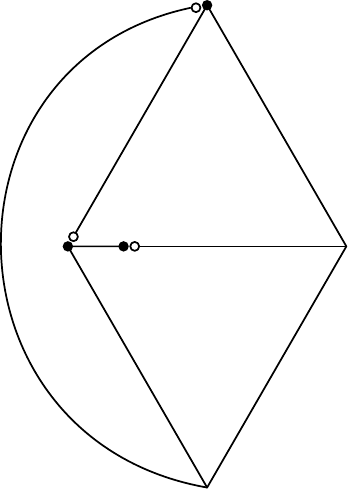}
\includegraphics[totalheight=0.18\textheight]{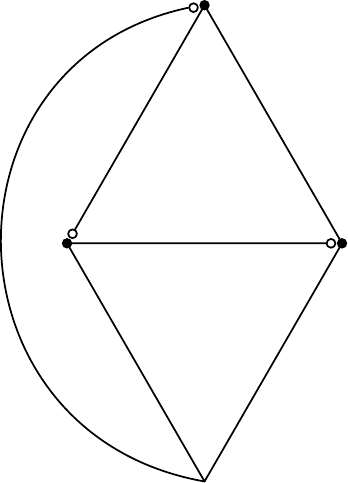}
\includegraphics[totalheight=0.18\textheight]{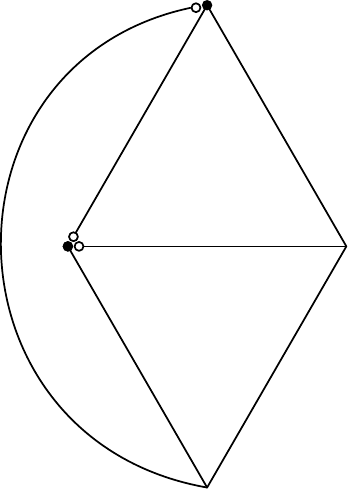}
\end{center}
  \caption{Examples of fundamental domains.}
  \label{ex03}
\end{figure}

An \emph{orientation} $\O$ of a metric graph $\Gamma$ is an equivalence class of pairs $(G,O)$, where $G$ is a model for $\Gamma$ and
$O$ is an orientation of the edges of $G$, where the equivalence relation is generated by the operation of replacing $G$ by a refinement $G'$
and letting $O'$ be the orientation induced by $O$.

For $q \in \Gamma$, we say that an orientation $\O$ is \emph{$q$-connected} if there is a directed path from $q$ to $p$ for every $p \in \Gamma$.

Given $q \in \Gamma$, there is a canonical way to associate a $q$-connected orientation $\O$ to any break set $\{ (p_1, \eta_1),\ldots,(p_g, \eta_g) \}$, as follows.
Choose $\varepsilon > 0$ small enough so that, setting $q_i = p_i+\varepsilon \eta_i$, we have (i) $\Gamma \backslash \{ q_1,\ldots, q_g \}$ is a tree, (ii) all $q_i$'s have valence $2$, and (iii) $q \neq q_i$ for any $i$.  We get an
associated $q$-connected orientation $\O_\varepsilon$ by orienting all edges away from $q$ in this tree, and letting $\varepsilon \rightarrow 0$ gives the desired $q$-connected orientation $\O$.
(In terms of the corresponding fundamental domain $(F,\psi)$ described in Remark~\ref{FundDomainRemark}, $\O$ corresponds to orienting all paths away from $\psi^{-1}(q)$ in the connected and simply connected space $F$.)

Conversely, there is a canonical way to associate a break divisor $D_{(\O,q)}$ to a $q$-connected orientation $\O$:

\[
D_{(\O,q)}:= (q)+D_\O = (q) + \sum_{p \in \Gamma} \left( {\rm indeg}_\O(p) - 1 \right)(p).
\]

\begin{lemma} \label{lem:break_vs_qorient}
Fix $q\in \Gamma$.  The map $\phi_q: \Div^g(\Gamma) \to \Div^{g-1}(\Gamma)$ sending $D$ to $D-(q)$ induces a bijection between break divisors and $q$-orientable divisors.
\end{lemma}

\begin{proof}

Let $D=(p_1)+\cdots+(p_g)$ be a break divisor and consider an associated break set $\{ (p_1, \eta_1),\ldots,(p_g, \eta_g) \}$ as in Lemma~\ref{lem:break_div_vs_pair}.
Let $\O$ be the associated $q$-connected orientation as above.

Let $p \in \Gamma$. If $p \ne q$ then ${\rm indeg}_\O(p) = 1+s(p)$, where $s(p)$ is the number of tangent directions $\eta_i$ with $p_i=p$,
and if $p=q$ then ${\rm indeg}_\O(p) = s(p)$.
Thus for every $p \in \Gamma$, the coefficient of $(p)$ in $D_\O+(q)$ is equal to $s(p)$, which is equal to $D(p)$ by construction. This proves that $\phi_q$ induces a map from break divisors to $q$-orientable divisors.

The map $\phi_q$ is clearly injective. To see that it is surjective, suppose $D_{\O}$ is a $q$-orientable divisor corresponding to a $q$-connected orientation $\O$.
We will equip the effective divisor $D_{\O}+(q)$ with a set of tangent directions $\eta$ so that $(D_{\O}+(q), \eta)$ is a break set.
By \emph{breaking} an edge $e$ adjacent to $p$, we will mean adding the tangent direction at $p$ which corresponds to the edge $e$ to the set $\eta$.
If ${\rm indeg}_\O(q) \geq 1$, break all the incoming edges at $q$.
For each $p \neq q$ with ${\rm indeg}_\O(p)\geq 2$, break all but one of the incoming edges at $p$. (The unbroken edge can be chosen arbitrarily.)
Then one easily checks that $|\eta|=g$ and $(D_{\O}+(q), \eta)$ is a break set.

\end{proof}

Note that the set of break directions is not uniquely determined by $\O$, so we do not get a bijection between $q$-connected orientations and break sets (see Figure~\ref{fig:nonunique}).

\begin{figure}[h!]
\begin{center}
\includegraphics[totalheight=0.18\textheight]{figures/breakpair5.pdf}
\includegraphics[totalheight=0.18\textheight]{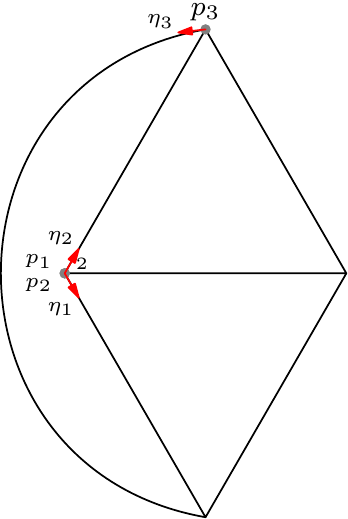}
\end{center}
\vspace{+20pt}
\begin{center}
\includegraphics[totalheight=0.18\textheight]{figures/breakfund5.pdf}
\includegraphics[totalheight=0.18\textheight]{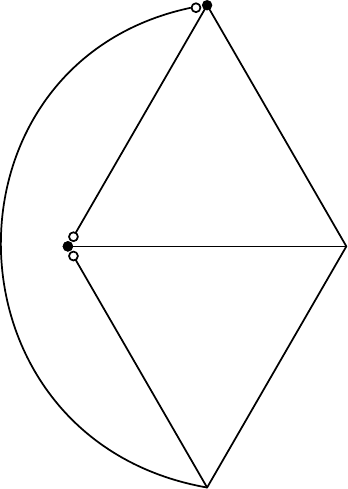}
\end{center}
\begin{center}
\includegraphics[totalheight=0.18\textheight]{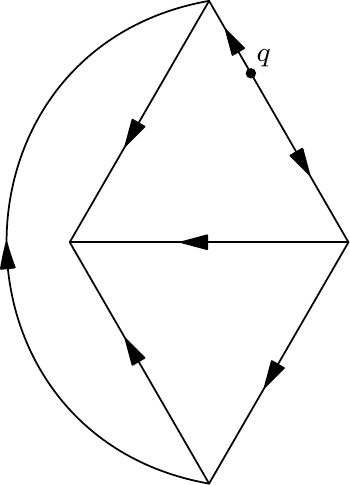}
\end{center}
  \caption{Two different break sets, their associated fundamental domains, and their identical associated $q$-connected orientation.}
  \label{fig:nonunique}
\end{figure}

\medskip

As a formal consequence of Lemma~\ref{lem:break_vs_qorient}, we have:

\begin{corollary}
Theorems \ref{thm:mainthm1} and \ref{thm:mainthm2} are equivalent
\end{corollary}

\subsection{The canonical cell decomposition of $\Pic^g(\Gamma)$}
Let $G$ be a model for $\Gamma$. Given any spanning tree $T$ of $G$ we let $\Sigma^{o}_T \subset \Sigma$ be the product of the interiors of all edges of $\Gamma$ not in $T$:
\[
\Sigma^{o}_T := \prod_{e \not \in E(T)} e^{o} \ .
\]

An element of $\Sigma^{o}_T$ determines the spanning tree $T$ of $G$ uniquely. Therefore for distinct spanning trees $T$ and $T'$ of $G$ we have $\Sigma^{o}_T \cap \Sigma^{o}_{T'} =\emptyset$. Since $\{e^{o}\}_{e \not \in E(T)}$ consists of mutually disjoint segments, no two elements of $\Sigma^{o}_T$ are identified under the action of the symmetric group and we may consider $\Sigma^{o}_T$ as an (open) \emph{paralleletope} inside $\Div_{+}^{g}(\Gamma)=\Gamma^{(g)}$.

\medskip

We let $\Sigma^{o} \subset \Div_{+}^{g}(\Gamma)$ be the (disjoint) union of the sets $\Sigma^{o}_T$ for all spanning trees $T$ of $G$:

\[
\Sigma^{o}:= \bigcup_{T}{\Sigma^{o}_T}
\]

Note that $\Sigma^{o}$,  unlike $\Sigma$, depends on the choice of the model $G$.

\medskip

Divisors in $\Sigma^{o}$ have the following useful property:

\begin{lemma} \label{lem:inj}
For any $D \in \Sigma^{o}$ we have $|D|=\{D\}$.
\end{lemma}

\begin{proof}
The proof is an application of least action principle (Lemma~\ref{LeastActLemma}). Let $D \in \Sigma^{o}_T$ for some spanning tree $T$ of the model $G$. Pick $p,q \in \Gamma$. Then one sees easily that $D$ is both $p$-reduced and $q$-reduced.
If $E=D+\Delta(f) \in |D|$ then by the least action principle we get $f(p) \leq f(q)$ and $f(q) \leq f(p)$. Thus $f$ is constant and $E=D$.
\end{proof}

Recall that the map $\pi: \Div_{+}^g(\Gamma) \rightarrow \Pic^g(\Gamma)$ is piecewise linear. It follows that the open cell $\C^{o}_T :=\pi(\Sigma^{o}_T)$ is also an (open) paralleletope inside the torus $\Pic^g(\Gamma)$. By Lemma~\ref{lem:inj}, the restriction of $\pi$ to each $\Sigma^{o}_T$ is injective with image $\C^{o}_T$. Moreover for distinct spanning trees $T$ and $T'$ of $G$ we have $\C^{o}_T \cap \C^{o}_{T'} =\emptyset$ because $\Sigma^{o}_T \cap \Sigma^{o}_{T'} =\emptyset$.
It follows that the restriction of the map $\pi$ to $\Sigma^{o} \subset \Div_{+}^{g}(\Gamma)$ is injective.

Let $\C^{o} \subset \Pic^{g}(\Gamma)$ be the (disjoint) union of the sets $\C^{o}_T$ for all spanning trees $T$ of $G$. Let $\C:=\overline{\C^{o}} =\bigcup_{T} \C_T$ and $\C_T:=\overline{\C^{o}_T}$ denote the topological closures of $\C^{o}$ and $\C^{o}_T$, respectively, inside $\Pic^{g}(\Gamma)$.
By Theorem \ref{thm:mainthm1} (stated in Section~\ref{sec:Introduction} and proved in Section~\ref{sec:orientations} below),
we have $\C =  \Pic^{g}(\Gamma)$.
Since $\Sigma = \bigcup \Sigma_T$ and $\Pic^g(\Gamma)$ are compact Hausdorff spaces, it also follows from Theorem \ref{thm:mainthm1} that
the canonical map from $\Sigma$ to $\Pic^g(\Gamma)$ is a homeomorphism.

\section{Orientations and divisors}
\label{sec:orientations}

\subsection{Finite graphs}
\label{sec:orientations.finite.graphs}
Assume $G$ is a finite, connected multigraph. As usual we denote the set of vertices by $V(G)$ and the set of edges by $E(G)$.
In what follows $S$ will always denote a subset of $V(G)$.
We denote by $G[S]$ the induced subgraph of $G$ with vertex set $S$.
Let $\chi(S)$ denote the topological Euler characteristic of $G[S]$, which is equal to the number of vertices of $G[S]$ minus the number of edges of $G[S]$.  (If $G[S]$ is connected then $\chi(S) = 1-g(S)$, where $g(S)$ is the \emph{genus}, or first Betti number, of $G[S]$.)

\medskip

Given a divisor $D \in \Div(G)$ we define
\[\chi(S,D)=\deg(D|_S)+\chi(S)  , \]
where $D|_S$ denotes the restriction of $D$ to $G[S]$, i.e., if $D=\sum_{p \in V(G)}{a_p(p)}$ then $D|_S=\sum_{p \in S}{a_p(p)}$.
For $S = V(G)$, we write $\chi(G,D)$ instead of $\chi(V(G),D)$.

\medskip

If $S$ and $T$ are disjoint subsets of $V(G)$, we define $e(S,T)$ to be the number of edges of $G$ with one end in $S$ and the other end in $T$.
We define $e(S)$ to be the number of edges both of whose endpoints belong to $S$.

\medskip

A \emph{submodular function} is function $f$ from the collection of subsets of $V(G)$ to $\RR$ such that for any subsets $S,T$ of $V(G)$, we have
\begin{equation} \label{eq:submod}
f(S) + f(T) \geq f(S\cup T) + f(S\cap T).
\end{equation}
If equality holds, then $f$ is called \emph{modular}. The following lemma is an immediate consequence of \eqref{eq:submod}.

\begin{lemma} \label{lem:intersection}
Subsets where a submodular function takes its minimum value are closed under intersection and union.
\end{lemma}

\begin{lemma}\label{lem:SubmodularLem}
For any divisor $D$, the function $\chi(\cdot, D)$ is submodular.
\end{lemma}

\begin{proof}
By definition, we have $\chi(S,D) = \deg(D|_S) + |S| - e(S)$.  One easily checks that
$f_1(S)=\deg(D|_S)$ and $f_2(S)=|S|$ are modular and that $f_3(S)=- e(S)$ is submodular, and the result follows.
\end{proof}

More precisely, we have the following quantitative refinement of submodularity:

\begin{lemma}\label{lem:EulerLem}
For any subsets $S,T$ of $V(G)$, we have
\begin{equation} \label{eq:strongsubmod}
\chi(S,D) + \chi(T,D) = \chi(S\cup T,D) + \chi(S \cap T,D) + e(S \backslash T, T \backslash S).
\end{equation}
In particular, if $S \cap T = \emptyset$ then
\begin{equation} \label{eq:EulerUnion}
\chi(S \cup T,D)=\chi(S,D) + \chi(T,D) -e(S,T).
\end{equation}
\end{lemma}

\begin{proof}
Since we have $\chi(S,D) = \deg(D|_S) + |S| - e(S)$, and both $\deg(D|_S)$ and $|S|$ are modular functions, it suffices to prove $e(S) + e(T) + e(S \backslash T, T \backslash S) = e(S \cap T) + e(S \cup T)$. This is a well-known stronger version of the submodularity of $e(\cdot)$.
\end{proof}

\medskip

For a given divisor $D \in \Div(G)$ we define

\[\chi_D = \min\{\chi(S,D): \, \emptyset \ne S \subsetneq V(G) \} \ ,\]

\[
\mathbf{S}(D) = \{ \emptyset \ne  S \subsetneq V(G) : \, \chi(S,D) = \chi_D \} \ .
\]

\medskip

\begin{corollary} \label{cor:minimal1}
Let $D \in \Div(G)$ be a divisor such that $\chi(G,D) \geq 0$ and $\chi_D < 0$. Then $\mathbf{S}(D)$ has a unique minimal element (with respect to inclusion).
\end{corollary}

\begin{proof}
Since $\chi(S,D)$ is a submodular function, by Lemma \ref{lem:intersection} there is a unique minimal subset $S$ for which $\chi(\cdot, D)$ takes on its minimum value, namely the intersection of all such subsets. Since $\chi (\emptyset,D) = 0$ and $\chi(G,D) \geq 0$, while $\chi_D < 0$, we conclude that the unique minimal element is a proper nonempty subset of $V(G)$, hence lies in $\mathbf{S}(D)$.
\end{proof}

\medskip

Recall from Section~\ref{sec:Introduction} that a divisor $D \in \Div(G)$ is called \emph{orientable} if there exists an orientation $\O$ on $G$ such that at every point $p \in V(G)$ we have $D(p)=\indeg_{\O} {(p)-1}$.
It is easy to see that if $D$ is orientable then $\deg(D)=g-1$, which is equivalent to saying that $\chi(G,D)=0$.
The following result strengthens this observation.\footnote{Although we discovered this result independently, Spencer Backman and the anonymous referee both informed us that Theorem~\ref{thm:orient} is in fact a reformulation of a classical result of Hakimi \cite{Hakimi} (see also \cite[Theorem 61.1]{Schrijver}).  Spencer Backman shows in \cite[Theorem 7.3]{Backman} that Theorem~\ref{thm:orient} is equivalent to the well-known Max-Flow Min-Cut Theorem.}

\begin{theorem}[Hakimi] \label{thm:orient}
A divisor $D \in \Div(G)$ is orientable if and only if $\chi(G,D)=0$ and $\chi(S,D) \geq 0$ for all non-empty subsets $S$ of $V(G)$.
(Equivalently, $D$ is orientable if and only if $\deg(D)=g-1$ and $\chi_D \geq 0$.)
\end{theorem}

\medskip

\begin{remark}
It follows from Lemma~\ref{lem:EulerLem} that one only needs to check the condition $\chi(S, D) \geq 0$ for subsets $S$ whose induced subgraph is \emph{connected}: if the condition is satisfied for all connected components of a set, it is automatically satisfied for the whole set by \eqref{eq:EulerUnion}.
\end{remark}

\medskip

\begin{theorem} \label{thm:eqiv_orient}
Every divisor $D \in \Div^{g-1}(G)$ is linearly equivalent to some orientable divisor.
\end{theorem}
\begin{proof}
If $D_0=D$ is not orientable then Theorem~\ref{thm:orient} guarantees that there exists a subset $\emptyset \ne S \subsetneq V(G)$ with $\chi(S,D) <0$. It follows from Corollary~\ref{cor:minimal1} that there is a unique minimal element (with respect to inclusion) in $\mathbf{S}(D_0) = \{ \emptyset \ne  S \subsetneq V(G) : \, \chi(S,D_0) \text{ is minimal}\}$. In other words, there is a unique non-empty set $S_0 \subsetneq V(G)$ such that $\chi(S_0,D_0)$ is minimal and which is contained in any other vertex set
$S$ with this property.

\medskip

Let $D_1$ be the divisor obtained from $D_0$ by simultaneously firing all vertices in $\bar{S_0}$.

{\bf Claim.} For all $\emptyset \ne S \subsetneq V(G)$ we have $\chi(S,D_1) \geq \chi(S_0,D_0)$. Moreover, if equality holds then $S \supsetneq S_0$.

\medskip

To prove this claim we consider the following cases:

\begin{itemize}
\item[(1)] $S \subsetneq S_0$.
In this case, $\chi(S,D_0) > \chi(S_0,D_0)$ (by the minimality of $S_0$) and $\chi(S,D_1) \geq \chi(S,D_0)$ (by the construction of $D_1$).

\medskip

\item[(2)] $S = S_0$.
In this case,  it follows from definitions that $\chi(S,D_1) = \chi(S_0,D_0)+e(S_0 , \bar{S_0}) > \chi(S_0,D_0)$.

\medskip

\item[(3)] $S \subseteq \bar{S_0}$.
In this case, it follows from definitions and \eqref{eq:EulerUnion} that
\[
\chi(S,D_1) = \chi(S,D_0)-e(S,S_0)=\chi(S \cup S_0,D_0)-\chi(S_0,D_0) \ .
\]
Therefore $\chi(S,D_1) \geq 0 >\chi(S_0,D_0)$.

\medskip

\item[(4)] $S \cap S_0 \ne \emptyset$ and $S \cap \bar{S_0} \ne \emptyset$.
In this case, let $A=S \cap S_0$ and $B=S \cap \bar{S_0}$. Then
\[
\chi(B,D_1)=\chi(B,D_0)-e(B,S_0)=\chi(B \cup S_0,D_0)-\chi(S_0,D_0) \geq 0
\]
and we have
\[
\begin{aligned}
\chi(S,D_1)&=\chi(A \cup B,D_1) \\
&= \chi(A,D_1)+\chi(B,D_1)- e(A,B)\\
& \geq \chi(A,D_1)- e(A,B)\\
& = \chi(A,D_0)+e(\bar{S_0}, A)- e(A,B) \\
& \geq \chi(A,D_0)\\
& \geq \chi(S_0, D_0) \ .
\end{aligned}
\]
By the minimality of $S_0$, equality can happen only if $A=S_0$, in which case $S =S_0 \cup B \supsetneq S_0$.
\end{itemize}

Let $S_1$ be the minimal element in $\mathbf{S}(D_1)$. It follows that either (i) $\chi(S_1,D_1) > \chi(S_0, D_0)$ or (ii) $\chi(S_1,D_1) = \chi(S_0, D_0)$ and $S_0\subsetneq S_1$. If we repeat this procedure we are therefore guaranteed to stop, at which point we will have an orientable divisor.

\end{proof}

There can be many different orientations associated to a given orientable divisor.  Also, in each equivalence class of divisors of degree $g-1$ there can be many different orientable divisors. Our next goal is to show that one can obtain a uniqueness result by fixing a vertex $q$.

\medskip

Recall from Section~\ref{sec:Introduction} that an orientation of $G$ is called \emph{$q$-connected} if there exists an oriented path from $q$ to each vertex $p$ of $G$.
Also, a divisor $D \in \Div(G)$ is called \emph{$q$-orientable} if there exists a $q$-connected orientation $\O$ on $G$ such that at every point $p \in V(G)$ we have $D(p)=\indeg_{\O} {(p)-1}$.

\begin{proposition} \label{prop:q_orient}
An orientable divisor $D \in \Div(G)$ is $q$-orientable if and only if $\chi(S,D) > 0$ for all non-empty subsets $S \subseteq V(G) \backslash \{q\}$.
\end{proposition}
\begin{proof}
If $D$ is $q$-orientable, it is in particular orientable and it follows from Theorem~\ref{thm:orient} that $\chi(S,D) \geq 0$ for all $\emptyset \ne S \subseteq V(G) \backslash \{q\}$. If $\chi(S,D) = 0$ for some subset $S$ then all edges connecting $S$ to $\bar{S}$ will be directed away from $S$ in any associated orientation, and the orientation will not be $q$-connected.

Now suppose that $D$ is an orientable divisor and that $\chi(S,D) > 0$ for all $\emptyset \ne S \subseteq V(G) \backslash \{q\}$. We will show that for any orientation associated to $D$ there is a directed path from $q$ to any other vertex. Let $p=p_1$ be an arbitrary vertex in $V(G)\backslash \{q\}$. Since $\chi(\{p_1\},D) > 0$ there exists an edge oriented towards $p_1$. Let $p_2$ be the other end of this oriented edge. If $p_2 = q$ we have found a directed path from $q$ to $p$. Otherwise $\{p_1,p_2\} \subseteq V(G) \backslash \{q\}$ and we have $\chi(\{p_1,p_2\},D) > 0$.
Continuing this procedure will yield a directed path from $q$ to $p$.
\end{proof}

\medskip

Fix a vertex $q$. For a $D \in \Div(G)$, we define
\[
\mathbf{S}_q(D) = \{ \emptyset \ne  S \subseteq V(G) \backslash \{q\} : \, \chi(S,D) = \chi_D \} \ .
\]

\medskip

\begin{lemma} \label{lem:maximal}
Fix a vertex $q$ and let $E \in \Div(G)$ be an orientable divisor, but not $q$-orientable. If $S_1, S_2 \in \mathbf{S}_q(E)$, then $S_1 \cup S_2 \in \mathbf{S}_q(E)$. In particular,
$\mathbf{S}_q(E)$ has a unique maximal element (with respect to inclusion).
\end{lemma}

\begin{proof}
When restricted to subsets of $V(G)\backslash \{q\}$, $\chi(\cdot, D)$ is still a submodular function. By Lemma \ref{lem:intersection}, 
there is a unique maximal subset of $V(G)\backslash \{q\}$ for which $\chi(S,D) = \chi_D$, namely the union of all such subsets. Since $E$ is orientable but not $q$-orientable, it follows from Theorem~\ref{thm:orient} and Proposition~\ref{prop:q_orient} that the maximal set is non-empty, hence lies in $\mathbf{S}_q(E)$.
\end{proof}

\begin{theorem} \label{thm:eqiv_q_orient}
Fix a vertex $q$. Every divisor $D \in \Div^{g-1}(G)$ is equivalent to a unique $q$-orientable divisor.
\end{theorem}
\begin{proof}
\emph{Existence. }By Theorem~\ref{thm:eqiv_orient} we know that $D \sim D_1$ for some orientable divisor $D_1$.
If $D_1$ is not already $q$-orientable,
let $S_1$ be the unique maximal element of $\mathbf{S}_q(D_1)$, which exists by Lemma~\ref{lem:maximal}. In any orientation associated to $D_1$ all edges
connecting $S_1$ to its complement are directed from $S_1$ to $\bar{S_1}$. Also, it follows from the maximality of $S_1$ that there is a directed path from $q$ to any vertex $p \in \bar{S_1}$. We now replace $D_1$ with the divisor $D_2$ obtained
by firing all vertices in the subset $\bar{S_1}$. This reverses the orientation of edges connecting $\bar{S_1}$ and $S_1$, directing them toward $S_1$, and leaves all other orientations unchanged. If $D_2$ is not already $q$-orientable, there exists a maximal element $S_2$ of $\mathbf{S}_q(D_2)$. Since there is a directed path from $q$ to any vertex $p \in \bar{S_1}$, as well as at least one vertex in $S_1$, it follows that $S_2$ is a proper subset of $S_1$. We now fire the subset $\bar{S_2}$ and proceed as before. This procedure must eventually terminate
in a $q$-orientable divisor.

\medskip

\emph{Uniqueness. } Let $D_1$ and $D_2$ be distinct orientable divisors and write $D_1 =D_2 + \Delta(f)$. Consider the (non-empty) set $X \subsetneq V(G)$ where $f$ achieves its (global) minimum value.  We have
\[
0 \leq \chi(X, D_1) \leq \chi(X, D_2) - e(X , \bar{X})= -\chi(\bar{X}, D_2) \leq 0 \ .
\]

It follows that $\chi(X, D_1)=\chi(\bar{X}, D_2) = 0$. This means that in any orientation associated to $D_1$ all edges are directed away from $X$, and in any orientation associated to $D_2$ all edges are towards $X$.  Thus there cannot be a vertex $q$ for which $D_1$ and $D_2$ are both $q$-orientable.

\end{proof}

\begin{remark} \label{rmk:algorithm}
It follows from the proofs of Theorem~\ref{thm:eqiv_orient} and Theorem~\ref{thm:eqiv_q_orient} that we have the following algorithm\footnote{Our algorithm for finding the unique $q$-orientable divisor equivalent to a given divisor $D$ takes exponential time.  Since this paper was first posted on the arXiv, Spencer Backman \cite{Backman} has given a polynomial-time algorithm for this problem, as well as for finding an associated orientation.}
 for finding the unique $q$-orientable divisor equivalent to a given divisor $D$.
\begin{itemize}
\item[(1)] While there exists a subset $\emptyset \ne S \subsetneq V(G)$ with $\chi(S,D) <0$,
find the unique minimal element $A$ of $\mathbf{S}(D)$
and replace $D$ with the divisor obtained by firing all vertices of $\bar{A}$.
\item[(2)] While there exists a subset $\emptyset \ne S \subseteq V(G) \backslash \{q\}$ with $\chi(S,D) = 0$, find the unique maximal element $B$ of $\mathbf{S}_q(D)$
and replace $D$ with
the divisor obtained from firing all vertices of $\bar{B}$.
\end{itemize}
It seems difficult to deduce effective algorithms for these problems from the work of Mikhalkin and Zharkov.
\end{remark}

\subsection{Metric graphs}

Fix a metric graph $\Gamma$ and a divisor $D \in \Div(\Gamma)$. A \emph{model for $(\Gamma, D)$} is a (weighted graph) model for $\Gamma$ whose vertex set contains the support of $D$. If a point $q \in \Gamma$ is also distinguished, we further assume that the vertex set of $G$ contains $q$.
\medskip

We call a (weighted graph) model $G$ for $\Gamma$ a \emph{semi-model for $(\Gamma, D)$} if  $\Gamma \backslash V(G)$ consists of a finite union of open intervals $\cup_{i=1}^r e_i^\circ$, and for each $i$ the set $e_i^\circ \cap \supp(D)$ is either empty or consists of a single point $p$ with $D(p)=1$. In other words, a semi-model is allowed to ``miss'' some points $p \in \supp(D)$ where $D(p)=1$ and $p$ is the only point of $\supp(D)$ lying in the corresponding open edge. Again, if a point $q \in \Gamma$ is also distinguished, we further assume that $q \in V(G)$.

\medskip

It turns out that semi-models are more convenient to work with than models when we want to show that certain algorithms for metric graphs terminate. We will obtain results similar to those in Section~\ref{sec:orientations.finite.graphs} by reducing the metric graph case to the case of finite graphs via semi-models.

\medskip

Let $G=(V, E)$ be an arbitrary semi-model for $(\Gamma, D)$, and let $G_D$ be the finite graph obtained from $G$ by removing all open edges which contain a point of the support of $D$. Note that $V(G_D) = V(G)$ and $E(G_D) \subset E(G)$. The following lemma will help us compare the set of orientable divisors on $\Gamma$ and on $G_D$.

\medskip

\begin{lemma}\label{lem:semi-model vs. metric graph}
Let $D\in \Div^{g-1}(\Gamma)$, and let $G$ be a semi-model for $(\Gamma, D)$. Let $D_G$ be the restriction of $D$ to $G$, i.e., if $D=\sum_{p\in \Gamma}a_p(p)$, then $D_G=\sum_{p\in V(G)} a_p(p)$. 
Then, for any $q \in \Gamma$, the divisor $D$ is ($q$-)orientable on $\Gamma$ if and only if $D_G$ is ($q$-)orientable on $G_D$.
\end{lemma}
\begin{proof}
  Suppose $D$ is given by the $q$-connected orientation $\O$ on $\Gamma$. Then $\O$ naturally induces a $q$-connected orientation\footnote{Strictly speaking, the induced orientation is not an orientation for the semi-model $G$ but for its refinement whose vertex set consists of $V(G)$ and the support of $D$.} on the semi-model $G$. On each edge $e\in E(G)\backslash E(G_D)$, the orientation $\O$ on $\Gamma$ must look like two arrows pointing toward a single point in the support of $D$. Hence, after removing these edges, $\O$ is still a $q$-connected orientation on $G_D$ and the resulting divisor is $D_G$.

  Conversely, given a $q$-connected orientation on $G_D$ for $D_G$, we obtain an orientation for $D$ on $\Gamma$ by directing every edge $e\in E(G)\backslash E(G_D)$ toward the corresponding point in the support of $D$.
\end{proof}

By Lemma \ref{lem:semi-model vs. metric graph}, in order to show that $D\in \Div(\Gamma)$ is $q$-orientable, it suffices to show that there is a semi-model $G$ for $D$ such that $D_G$ is $q$-orientable on $G_D$. This helps us reduce our problems to the case of finite graphs. 

\begin{remark}
In the following, since we are working with different finite graphs, we use the notation $\chi_G( S , D)$ ,  $\mathbf{S}(G ,D)$, and $\mathbf{S}_q(G ,D)$ (instead of $\chi( S , D)$ ,  $\mathbf{S}(D)$, and $\mathbf{S}_q(D)$) to identify the underlying graph we are working with at each step.
\end{remark}

\begin{theorem}\label{thm:metric equiv}
Every divisor $D\in \Div^{g-1}(\Gamma)$ is equivalent to an orientable divisor on $\Gamma$. More precisely, let $G$ be a model for $(\Gamma, D)$. Then $D$ is equivalent to a divisor $D'$ on $\Gamma$ such that $G$ is a semi-model for $D'$ and $D'_G$ is orientable on $G_{D'}$.
\end{theorem}

\begin{proof}
 Fix a model $G$ for $(\Gamma, D)$. Let $D^0=D$, let $k \geq 0$, and assume that $G$ is a semi-model for $D^k$.  
  If $D^k$ is not orientable, we inductively define a divisor $D^{k+1} \in \Div^{g-1}(\Gamma)$ equivalent to $D$ as follows. 
  By Lemma~\ref{lem:semi-model vs. metric graph}, Corollary \ref{cor:minimal1}, and Theorem \ref{thm:orient}, there is a unique
  minimal element $S_k \in \mathbf{S}(G_{D^k} ,D^k)$. Define $\ell$ to be the minimal distance between $\Gamma[S_k]$ (the closed subset of $\Gamma$ corresponding to the induced subgraph $G[S_k]$) and $T_k := \left( V(G)\cup \supp(D^k) \right) \backslash \Gamma[S_k]$. Let $C$ be the cut consisting of all the closed intervals connecting $S_k$ to points of $T_k$, and fire by moving each chip on an interval in $C$ a distance $\ell$ toward $S_k$. We call the resulting divisor $D^{k+1}$.

Clearly, $G$ is again a semi-model for $D^{k+1}$. The claim in the proof of Theorem \ref{thm:eqiv_orient} also holds here (with a similar proof), i.e., for each $k\geq 0$, either (i) $\chi_{G_{D^{k+1}}}(S_{k+1},D^{k+1}) > \chi_{G_{D^k}}(S_k, D^k)$ or (ii) $\chi_{G_{D^{k+1}}}(S_{k+1},D^{k+1}) = \chi_{G_{D^k}}(S_k, D^k)$ and $S_k\subsetneq S_{k+1}$. Therefore this procedure is guaranteed to stop, at which point we will have an orientable divisor.
\end{proof}

\medskip

\begin{theorem} \label{thm:qconn_unique}
Every divisor $D \in \Div^{g-1}(\Gamma)$ is equivalent to a unique $q$-orientable divisor.
\end{theorem}
\begin{proof}
 
\emph{Existence.} By Theorem \ref{thm:metric equiv}, we may assume that $D=D^0$ is orientable. Fix a model $G$ for $(\Gamma, D)$ and apply the following algorithm:

For $k\geq 0$, assume $G$ is a semi-model for $D^k$.  If $D^k$ is not already $q$-orientable, then by Proposition \ref{prop:q_orient} and Lemma \ref{lem:maximal} there is a unique maximal element $S_k \in \mathbf{S}_q(G_{D^k} ,D^k)$. Define $\ell$ and $C$ exactly as in the proof of Theorem \ref{thm:metric equiv}, and fire by moving each chip on an interval in $C$ a distance $\ell$ toward $S_k$.   Clearly, $G$ is also a semi-model for the resulting divisor $D^{k+1}$.

In obtaining $D^{k+1}$ from $D^k$, at least one chip must arrive at some vertex $v \in S_k$, so there is a directed path from $q$ to $v$ in the 
corresponding orientation $\O_{k+1}$. Thus $S_{k+1}$ is a proper subset of $S_k$, and the algorithm will terminate to give a $q$-orientable divisor.

\medskip

\emph{Uniqueness.} This is identical to the proof of uniqueness in Theorem~\ref{thm:eqiv_q_orient}, starting with distinct orientable divisors 
$D_1$ and $D_2$ and letting $G$ be a model for $\Gamma$ such that $V(G)$ contains $\supp(D_1)\cup \supp(D_2)$. 
(Note that if $D_1 =D_2 + \Delta(f)$ then $f$ is linear on every edge of $G$ because $\Delta(f)=D_1-D_2$ is supported on $V(G)$.)
\end{proof}


\subsection{Break divisors and universally reduced divisors}

Break divisors, like $q$-reduced divisors, provide a way for us to pick out a distinguished representative from each linear equivalence class of divisors.
In this section we link the two notions by characterizing break divisors as limits of degree $g$ effective divisors which are $q$-reduced for all $q \in \Gamma$.

\medskip

\begin{lemma} \label{lem:rigid}
An effective divisor $D$ on $\Gamma$ is $q$-reduced for every $q \in \Gamma$ if and only if $|D|=\{D\}$.
\end{lemma}

\begin{proof}
If $D$ is  $q$-reduced for every $q \in \Gamma$, it follows from the least action principle that $|D|=\{D\}$, exactly as in the proof of Lemma~\ref{lem:inj}.
Conversely, if $D$ is effective then for every $q$ the unique $q$-reduced divisor $D_q$ equivalent to $D$ is also effective, so if $|D|=\{ D \}$ we must have $D_q=D$.
\end{proof}

We define a divisor on $\Gamma$ to be \emph{universally reduced} if it has degree $g$, is effective, and is $q$-reduced for every $q \in \Gamma$.

\begin{theorem}
The set $\Sigma$ of break divisors on $\Gamma$ is equal to the closure in $\Div^g_+(\Gamma)$ of the set of universally reduced divisors.
\end{theorem}

\begin{proof}
Let $\Omega$ be the set of universally reduced divisors.
By Lemma~\ref{lem:inj}, $\Sigma$ contains a dense subset $\Sigma^\circ$ belonging to $\Omega$.
Since $\Sigma$ is compact, we have $\Sigma \subseteq \bar\Omega$.
To prove the reverse inclusion, note that if $D \in \Omega$ then by Lemma~\ref{lem:rigid} we have $|D|=\{D \}$.
Since $D$ is equivalent to a break divisor by Theorem~\ref{thm:mainthm1}, $D$ must itself be a break divisor.
Thus $\Omega \subseteq \Sigma$, and by taking closures we obtain $\bar\Omega \subseteq \Sigma$ as desired.
\end{proof}

\subsection{Integral break divisors on finite graphs}
\label{sec:integralbreakdivisors}

Suppose $G$ is a finite (unweighted) graph and that $\Gamma$ is the associated metric graph in which all edges of $G$ are assigned length $1$.
We let $\Sigma(G) := \Sigma \cap \Div^g_+(G)$ denote the set of \emph{integral break divisors}, i.e., those break divisors which are supported on vertices of $G$.
Inside $\Pic^d(\Gamma)$, for each integer $d$, we have the finite subset $\Pic^d(G)$ consisting of linear equivalence classes of divisors of degree $d$ supported on the vertices of $G$.
For $d=0$ the set $\Pic^0(G)$ is a group whose cardinality equal to the number of spanning trees in $G$, and each $\Pic^d(G)$ is a torsor for $\Pic^0(G)$.

\begin{theorem}
\label{thm:integralbreakdivisorthm}
The canonical map $\pi : \Div^g_+(\Gamma) \to \Pic^g(\Gamma)$ induces a bijection from $\Sigma(G)$ to $\Pic^g(G)$.
In particular, the number of integral break divisors is equal to the number of spanning trees of $G$.
\end{theorem}

\begin{proof}
Choose $q \in V(G)$. Then the result follows from Theorem~\ref{thm:eqiv_q_orient}, which says that every element of $\Div^{g-1}(G)$ is linearly equivalent to an \emph{integral} divisor of the form $D_{\mathcal O}$ with ${\mathcal O}$
a $q$-connected orientation, together with the equivalence of Theorem~\ref{thm:mainthm1} and Theorem~\ref{thm:mainthm2}.
\end{proof}

\begin{remark}
One can interpret Theorem~\ref{thm:integralbreakdivisorthm} as follows:
although the set of spanning trees of $G$ is not canonically a torsor for $\Pic^0(G)$, the set $\Sigma(G)$ of integral break divisors is.
Fixing a vertex $q$ of $G$ gives a bijection between $\Pic^0(G)$ and $\Pic^g(G)$, and for a generic choice of $\lambda \in \Jac(\Gamma)$ there will be exactly one element of
$\Pic^g(G)+\lambda$ in each open cell $\C^\circ_T$.  Thus the pair $(q, \lambda)$ provides a bijection between elements of $\Pic^0(G)$ and spanning trees.
\end{remark}



\section{The dual Matrix-Tree Theorem}

\subsection{Weights of spanning trees and volumes}

Let $\Gamma$ be a metric graph and let the weighted graph $G$ be a model for $\Gamma$. Given any spanning tree $T$ of $G$ we define the \emph{weight} of $T$ to be the product of the lengths of all edges of $G$ {\bf not in} $T$:
\[
w(T) := \prod_{e \not \in E(T)}{\ell(e)} \ .
\]

We also define
\begin{equation} \label{wG}
w(G) := \sum_{T}{w(T)} \ ,
\end{equation}
the sum being over all spanning trees of $G$.
It is easy to check that $w(G)$ is invariant under refinement, and therefore depends only on the underlying metric graph $\Gamma$.
We will therefore also denote $w(G)$ by $w(\Gamma)$.

\medskip

Let $\Lambda$ be a \emph{lattice}, i.e., a free $\ZZ$-module of some rank $g$ equipped with a symmetric integer-valued bilinear form $\langle \cdot, \cdot \rangle$
whose corresponding quadratic form is positive definite on $\Lambda_{\RR} := \Lambda \otimes \RR$.
We denote by ${\rm Vol}(\Lambda)$ the volume of any fundamental domain for $\Lambda$, or equivalently, the volume of the the real torus $\Lambda_{\RR}/\Lambda$.
If $M$ is any {\rm Gram matrix} for $\Lambda$, i.e., the matrix $(\langle \lambda_i, \lambda_j \rangle)$, where $\lambda_1,\ldots,\lambda_g$ is a basis for $\Lambda$,
then it is well known that ${\rm Vol}(\Lambda)=\sqrt{\det(M)}$.

\medskip

Our goal in this section is to give a geometric proof, via the decomposition of $\Pic^g(\Gamma)$ into the cells $C_T$, of the following ``dual version'' of
Kirchhoff's Matrix-Tree Theorem.  

\begin{theorem}[Matrix-Tree Theorem, Dual Version] \label{thm:DualKirchhoff}
The volume of the real torus $\Jac(\Gamma)$ is ${\rm Vol}(\Jac(\Gamma)) = \sqrt{w(\Gamma)}$.
Equivalently, if $M$ is any Gram matrix for the cycle lattice $H_1(\Gamma,\ZZ)$ then $\det(M) = w(\Gamma)$.
\end{theorem}

One can explicitly calculate a Gram matrix for $H_1(\Gamma,\ZZ)$ as follows.
Fix an arbitrary orientation of the model $G$ and a spanning tree $T$ of $G$.
For each $e \not \in T$, the \emph{fundamental cycle} $\cb(T,e)$ is the unique element of $H_1(G,\ZZ)$ contained in $T \cup e$.
It is well known that the set
\[
 \Cb (T):= \{\cb(T,e) : \, e \in E(G) \backslash E(T)\}
\]
of fundamental cycles associated to $T$ forms a basis for $H_1(\Gamma,\ZZ)$.

\medskip


\medskip

Let $m= |E(G)|$ and let $g=|E(G)|-|V(G)|+1$ be the rank of $H_1(\Gamma,\ZZ)$.  Let $\Cf_T$ denote the $g \times m$ matrix whose rows correspond to basis elements $\cb(T,e) \in  \Cb (T)$. For this, first fix a labeling $\{e_1, e_2, \cdots , e_m\}$ of $E(G)$. The $(i,j)$-entry of $\Cf_T$ is
\[
\begin{cases}
+\sqrt{\ell(e_j)} &\text{if $+e_j$ appears in $\cb(T,e_i)$} ;\\
-\sqrt{\ell(e_j)} &\text{if $-e_j$ appears in $\cb(T,e_i)$} ;\\
0 &\text{otherwise} .
\end{cases}
\]

Then $\Cf_T \Cf_T^t$ is a Gram matrix for $H_1(\Gamma,\ZZ)$ and consequently for any spanning tree $T$ we have
\begin{equation} \label{volJacDet}
{\rm Vol}(\Jac(\Gamma))=\sqrt{\det(\Cf_T\Cf_T^t)} \ .
\end{equation}

Alternatively let $\Df$ denote the $m \times m$ diagonal matrix whose $(i,i)$-entry is $\sqrt{\ell(e_i)}$. Then $\Cf_T=\Cf'_T \Df$, where $\Cf'_T$ is the matrix whose $(i,j)$-entry is
\[
\begin{cases}
+1 &\text{if $+e_j$ appears in $\cb(T,e_i)$} ;\\
-1 &\text{if $-e_j$ appears in $\cb(T,e_i)$} ;\\
0 &\text{otherwise} .
\end{cases}
\]

Fix an identification of $\Jac(\Gamma)$ with $\Pic^g(\Gamma)$ and let $\D_T$ be the cell in $\Jac(\Gamma)$ corresponding to the cell $C_T$ in $\Pic^g(\Gamma)$.
In order to prove Theorem~\ref{thm:DualKirchhoff}, it suffices to prove the following result:

\begin{proposition} \label{volume}
 ${\rm Vol}(\D_T ) = w(T) /\sqrt{\det(\Cf_T\Cf_T^t)} = w(T) / {\rm Vol}(\Jac(\Gamma))$ .
\end{proposition}

\begin{proof}
Let $\tilde{e}$ denote the orthogonal projection of an oriented edge $e$ in $C_1(\Gamma, \RR)$ onto $H_1(\Gamma, \RR)$.
Then the volume of $\D_T$ is equal to the $\sqrt{\det(\Lf_T \Lf_T^{t})}$, where $\Lf_T$ is the $g \times m$ matrix whose rows correspond to the basis elements $\tilde{e}$ for $e \not \in T$.

Let $\Df_T$ denote the $g \times g$ diagonal matrix whose $(i,i)$-entry is $\sqrt{\ell(e_i)}$ for $e_i \not \in T$. Then $\Lf_T = \Df_T^2 (\Cf_T \Cf_T^{t})^{-1} \Cf_T$ since
\[
\Lf_T \Cf_T^{t} = \Df_T^2 (\Cf_T \Cf_T^{t})^{-1} (\Cf_T \Cf_T^{t}) = \Df_T^2 \ .
\]

Now we also have
\[
\Lf_T \Lf_T^{t}= \Df_T^2 (\Cf_T \Cf_T^{t})^{-1} (\Cf_T \Cf_T^{t}) (\Cf_T \Cf_T^{t})^{-1} \Df_T^2 =\Df_T^2  (\Cf_T \Cf_T^{t})^{-1} \Df_T^2
\]
and therefore
\[
{\rm Vol}(\D_T ) =
\sqrt{\det(\Df_T^2  (\Cf_T \Cf_T^{t})^{-1} \Df_T^2)}=\det(\Df_T)^2 / \sqrt{\det(\Cf_T\Cf_T^t)}= w(T) / \sqrt{\det(\Cf_T\Cf_T^t)} \ .
\]
\end{proof}

\subsection{Matroid duality}

We now explain the precise sense in which Theorem~\ref{thm:DualKirchhoff} is dual to the usual version of
Kirchhoff's Matrix-Tree Theorem.  In order to do this, we first give a slightly more canonical formulation of the latter.

\medskip

Given any spanning tree $T$ of $G$, we define the \emph{coweight} of $T$ to be the product of the lengths all edges of $G$ {\bf in} $T$:
\[
w'(T) := \prod_{e \in E(T)}{\ell^{-1}(e)}.
\]

We also define
\[
w'(G) := \sum_{T}{w'(T)}
\]
to be the sum of $w'(T)$ over all spanning trees of $G$.

Note that, unlike $w(G)$, the quantity $w'(G)$ is not invariant under refinement and is therefore \emph{not} an invariant of the metric graph $\Gamma$.

\begin{theorem}[Kirchhoff's Matrix-Tree Theorem, Canonical Version]
\label{thm:Kirch}
Let $B$ be the \emph{cocycle lattice} (or \emph{lattice of integer cuts}) of $G$ .
Then ${\rm Vol}(B)^2 = w'(G)$.
Equivalently, if $M'$ is any Gram matrix for $B$ then $\det(M')=w'(G)$.
\end{theorem}

If we fix a vertex $q$ of $G$, then the \emph{reduced Laplacian matrix} $Q' = Q_q$ obtained by deleting the row and column corresponding to $q$ in the usual
weighted Laplacian matrix for $G$ is the Gram matrix of the basis for $B$ consisting of $\partial^*(p)$ for vertices $p \neq q$, where $\partial^* : C^0(G,\ZZ) \to C^1(G,\ZZ)$
is adjoint to the usual boundary map $\partial : C_1(G,\ZZ) \to C_0(G,\ZZ)$.
We therefore obtain:

\begin{corollary}[Kirchhoff's Matrix-Tree Theorem, Usual Version]
\label{cor:Kirch}
Fix $q \in V(G)$ and let $Q'$ be the corresponding reduced Laplacian matrix.  Then $\det(Q')=w'(G)$.
\end{corollary}

Theorem~\ref{thm:Kirch} is dual to Theorem~\ref{thm:DualKirchhoff} in the precise sense that it is obtained by interchanging the cycle lattice with the cocycle lattice and weights
with coweights.  As is well known, interchanging the cycles and cocycles in a graph is a special case of matroid duality.

\begin{remark} \label{CBRemark}
The classical linear-algebraic proof of Kirchhoff's Matrix-Tree Theorem is an application of the Cauchy-Binet formula.
One can also prove Theorem~\ref{thm:DualKirchhoff} via the Cauchy-Binet formula; we omit the details.
We note in addition that a generalization of the computations in \cite[Lemma 3.4]{KotaniSunada} \cite[Lemma 2]{ChenYe} to the setting of weighted graphs can
be used to prove that Kirchhoff's Matrix-Tree Theorem and its dual version are equivalent.
\end{remark}

\def\soft#1{\leavevmode\setbox0=\hbox{h}\dimen7=\ht0\advance \dimen7
  by-1ex\relax\if t#1\relax\rlap{\raise.6\dimen7
  \hbox{\kern.3ex\char'47}}#1\relax\else\if T#1\relax
  \rlap{\raise.5\dimen7\hbox{\kern1.3ex\char'47}}#1\relax \else\if
  d#1\relax\rlap{\raise.5\dimen7\hbox{\kern.9ex \char'47}}#1\relax\else\if
  D#1\relax\rlap{\raise.5\dimen7 \hbox{\kern1.4ex\char'47}}#1\relax\else\if
  l#1\relax \rlap{\raise.5\dimen7\hbox{\kern.4ex\char'47}}#1\relax \else\if
  L#1\relax\rlap{\raise.5\dimen7\hbox{\kern.7ex
  \char'47}}#1\relax\else\message{accent \string\soft \space #1 not
  defined!}#1\relax\fi\fi\fi\fi\fi\fi}

\end{document}